\documentclass[reqno, oneside, 11pt]{amsart}
\usepackage[letterpaper]{geometry}
\usepackage{appendix}
\usepackage{amsmath}
\usepackage{amssymb}
\usepackage{enumerate,enumitem}
\usepackage{mathrsfs}
\usepackage{mathtools}
\usepackage{hyperref}
\usepackage{amsthm, mathdots, multicol,tcolorbox}
\usepackage{xcolor}

\DeclareMathOperator{\Jac}{Jac}
\DeclareMathOperator{\USp}{USp}
\DeclareMathOperator{\SU}{SU}
\DeclareMathOperator{\Sp}{Sp}

\DeclareMathOperator{\U}{U}
\DeclareMathOperator{\GL}{GL}

\DeclareMathOperator{\ST}{ST}
\DeclareMathOperator{\Gal}{Gal}
\DeclareMathOperator{\End}{End}
\DeclareMathOperator{\Aut}{Aut}
\DeclareMathOperator{\diag}{diag}
\DeclareMathOperator{\antidiag}{antidiag}
\DeclareMathOperator{\TL}{TL}
\DeclareMathOperator{\AST}{AST}

\DeclareMathOperator{\LL}{L}
\DeclareMathOperator{\tr}{tr}
\DeclareMathOperator{\Tr}{Tr}

\DeclareMathOperator{\rk}{rk}
\DeclareMathOperator{\NS}{NS}

\newtheorem{theorem}{Theorem}[section]
\newtheorem{example}[theorem]{Example}
\newtheorem{proposition}[theorem]{Proposition}

\newtheorem{lemma}[theorem]{Lemma}
\newtheorem{definition}[theorem]{Definition}
\newtheorem{corollary}[theorem]{Corollary}
\newtheorem{conjecture}[theorem]{Conjecture}

\theoremstyle{remark}
\newtheorem{remark}{Remark}
\newtheorem*{remark*}{Remark}

\author{Heidi Goodson}
\address{Department of Mathematics, Brooklyn College, City University of New York; 2900 Bedford Avenue, Brooklyn, NY 11210 USA}
\email{heidi.goodson@brooklyn.cuny.edu}

\title[Sato-Tate Distributions of  Catalan Curves]{Sato-Tate Distributions of  Catalan Curves}

\begin{document}

\begin{abstract}
For distinct odd primes $p$ and $q$, we define the Catalan curve $C_{p,q}$ by the affine equation $y^q=x^p-1$. In this article we construct the Sato-Tate groups of the Jacobians in order to study the limiting distributions of coefficients of their normalized L-polynomials. Catalan Jacobians are nondegenerate and simple with noncyclic Galois groups (of the endomorphism fields over $\mathbb Q$), thus making them interesting varieties to study in the context of Sato-Tate groups. We compute both statistical and numerical moments for the limiting distributions. Lastly, we determine the Galois endomorphism types of the Jacobians using both old and new techniques. 
\end{abstract}

\subjclass[2010]{11M50, 11G10, 11G20, 14G10}
\keywords{Sato-Tate groups, Sato-Tate distributions, Jacobian varieties, endomorphism algebras}

\maketitle

\section{Introduction}\label{sec:intro}

Let $p$ and $q$ be distinct odd primes. The nonsingular genus $(p-1)(q-1)/2$ curve $C_{p,q}$ defined by the affine equation
$$y^q=x^p-1$$
is called \emph{Catalan curve}, likely named after the famous Catalan conjecture\footnote{The Catalan conjecture was proved by Mih\u{a}ilescu \cite{Mihailescu2004} more than 150 years after Catalan published his conjecture.} regarding consecutive integers that are perfect powers.  The primary goals of this article are to study  the limiting distributions of coefficients of normalized L-polynomials of the Jacobians of Catalan curves and to study their real endomorphism algebras. These types of distributions are called {\it Sato-Tate distributions}, named for Mikio Sato and John Tate who independently made conjectures for those associated to elliptic curves in the 1960s. Their conjectures are known to be true for elliptic curves without complex multiplication defined over totally real fields and for all elliptic curves with complex multiplication. 

Recent work has expanded this field to the study of distributions for abelian varieties of dimension at least 2. The generalized Sato-Tate conjecture  predicts that the distributions converge to the distributions of traces in a compact Lie group referred to as the Sato-Tate group. The Sato-Tate group is related to the Mumford-Tate group, Hodge group, motivic Galois group, and $l$-adic monodromy group, thus placing this area of research at the intersection of many important fields in number theory, group theory, and algebraic geometry. 

Before describing the work in higher genus, we first recall the original Sato-Tate conjecture for elliptic curves.  Let $F$ be a number field,  $E/F$ be an elliptic curve without complex multiplication (CM), and $v$ be a finite prime of $F$ such that $E$ has good reduction at $v$.  By a theorem of Hasse,  the number of ${F}_{q_v}$ points of $E$ is $q_v+1-a_v$, where ${F}_{q_v}$ denotes the residue field of $v$ and  $a_v$ is an integer (called the trace of Frobenius) satisfying $|a_v| \leq 2{q_v}^{1/2}$.  The Sato-Tate conjecture predicts that,  as $v$ varies through the primes of good reduction for $E$, the normalized Frobenius traces $a_v/{q_v}^{1/2}$ are distributed in the interval $[-2,2]$ with respect to the image of the Haar measure on the special unitary group $\SU(2)$. This conjecture  was proved in 2008 for non-CM elliptic curves  defined over totally real fields (see \cite{Barnet2011, Clozel2008,Harris2010,Taylor2008}). 
 
The distributions of the normalized Frobenius traces were known much earlier for CM elliptic curves over all fields: they are distributed with respect to the image of the Haar measure on either the  unitary group $\U(1)$ or the normalizer of $\U(1)$ in $\SU(2)$, depending on whether or not the field of definition contains the field of complex multiplication (see the exposition in \cite{Banaszak2015}).  

The generalized Sato-Tate conjecture for an abelian variety predicts the existence of a compact Lie group that determines the limiting distribution of normalized local Euler factors (see, for example, \cite{SutherlandNotes}).  In our work, the abelian varieties will be the Jacobians of curves, and so we now state the conjecture  specifically for Jacobian varieties. 

Let $C$ be a smooth, projective, genus $g$ curve defined over a number field $F$. The Sato-Tate group of the Jacobian of $C$, $\ST(\Jac(C))\subseteq \USp(2g)$, is a compact Lie group satisfying the following property: for each prime $\mathfrak p$ at which $C$ has good reduction, there exists a conjugacy class of $\ST(\Jac(C))$ whose characteristic polynomial equals the normalized $\LL$-polynomial at $\mathfrak p$
 \begin{equation}\label{eqn:normLpoly}\overline L_{\mathfrak p}(C,T) = T^{2g}+a_1T^{2g-1} +a_2 T^{2g-2} + \cdots + a_2T^2+a_1T+1.\end{equation}

\begin{conjecture}\label{conjec:generalST}
(Generalized Sato-Tate Conjecture)
 Let  $(x_{\mathfrak p})$ be the sequence  of conjugacy classes of normalized images of Frobenius elements in $\ST(\Jac(C))$ at primes $\mathfrak p$ of good reduction for $\Jac(C)$, ordered by norm.  Then the sequence $(x_{\mathfrak p})$ is equidistributed   with respect to the pushforward  of the Haar measure of $\ST(\Jac(C))$ to its space of conjugacy classes.
\end{conjecture}

The generalized Sato-Tate conjecture was proved for CM abelian varieties in \cite{Joh17}.  Beyond proving the generalized Sato-Tate conjecture, it is also important to determine the explicit embedding of the Sato-Tate group   inside $\USp(2g)$ in order to provide an explicit description of the limiting distributions of the normalized $\LL$-polynomial. The distributions can also give us information about certain arithmetic invariants, such as the real endomorphism algebra and the N\'eron-Severi group of the abelian variety. 

Determining these Sato-Tate groups is the source of ongoing interest and work. The current literature contains many articles for abelian varieties of fixed, small genus. For example, \cite{FKS2012} and \cite{FiteKS_genus3_2019, FKS2021satotate} determine all possible Sato-Tate groups in dimension 2 and 3, respectively, by determining which subgroups of the unitary symplectic group satisfy certain axioms (see Section \ref{sec:SatoTateAxioms}). Other articles determine Sato-Tate groups for certain families of genus 2 and 3 curves (see \cite{FS2016, LarioSomoza2018}) or for twists of curves (see \cite{Arora2016, FiteLorenzoSutherland2018, FiteSuthlerland2014}). 

As noted in \cite{FKS2021satotate}, it is not expected, in dimension greater than 3, that every  group satisfying the Sato-Tate axioms (see Section \ref{sec:SatoTateAxioms}) can be realized using Jacobians of curves. There are currently two articles that study families of Jacobian varieties of arbitrarily high genus (see \cite{EmoryGoodson2020,FiteGonzalezLario2016}).  In this article we provide a new example of an infinite family of Sato-Tate groups that can be realized by higher dimensional Jacobian varieties. Unlike the earlier articles, the component groups of these Sato-Tate groups are products of cyclic groups. Furthermore, we go a step further  by computing the Galois endomorphism types of the Jacobians.

\subsection*{Organization of the paper}

In Section \ref{sec:background} we provide an overview of this area of research, as well as some results that will be applied when computing the Sato-Tate groups of Jacobians of Catalan curves (Catalan Jacobians). We begin the section with a precise definition of the algebraic Sato-Tate group, which is needed in order to define the Sato-Tate group of an abelian variety. We then discuss nondegeneracy and its relationship to the Hodge conjecture in Section \ref{sec:nondegeneracy}. In Section \ref{sec:Lefschetz} we remind readers of the definition of the twisted Lefschetz group. We prove in Proposition \ref{prop:AST=TL} that, for Catalan Jacobians, this group is equal to the algebraic Sato-Tate group. We end the background section with a discussion of the Sato-Tate axioms in Section \ref{sec:SatoTateAxioms}.

In Section \ref{sec:EndoGaloisAction} we define a generator for the group of endomorphisms of the Catalan Jacobian. Catalan curves have CM by the field $\mathbb Q(\zeta_{pq})$, and we study the action of  $\Gal(\mathbb Q(\zeta_{pq})/\mathbb Q)$ on the endomorphism. We give complete descriptions of the actions in Propositions \ref{prop:sigma_q} and \ref{prop:sigma_p}. These descriptions are needed in order to determine the twisted Lefschetz groups of Catalan Jacobians.

The main goal of Section \ref{sec:SatoTateGroup} is to compute the Sato-Tate groups of  Catalan Jacobians.  Since the algebraic Sato-Tate conjecture for Catalan Jacobians is explained by endomorphisms, the component group of the Sato-Tate group is isomorphic to the Galois group $\Gal(\mathbb{Q}(\zeta_{pq})/\mathbb{Q})$. This immediately tells us that the component group is a product of cyclic groups of orders $p-1$ and $q-1$. Still, we give explicit generators for the component group in Theorem \ref{thm:STgroupqsmaller} in order to obtain the  limiting distribution of normalized local Euler factors.

There is  an additional interesting result in Section \ref{sec:SatoTateGroup}. Lemma \ref{lemma:BuildGamma}, which we use to determine component group generators, can be applied more broadly than the family of Jacobian varieties that we study in this paper. For any nondegenerate abelian variety with  endomorphism given by a diagonal matrix and with CM by a cyclotomic field, this lemma can be used to determine the component group generators. This may seem like a very specific scenario, but we note that this describes the Jacobians considered in, for example, \cite{EmoryGoodson2020, FiteGonzalezLario2016, LarioSomoza2018} and \cite{FS2016} (for certain values of $c$).

In Section \ref{sec:moments}, we compute moment statistics associated to the Sato-Tate groups of Catalan Jacobians. These moment statistics can be used to verify the equidistribution statement of the generalized Sato-Tate conjecture by comparing them to moment statistics obtained for the traces $a_i$ in the normalized $L$-polynomial. The numerical moment statistics are an approximation since one can only ever compute them up to some prime. We list moments for some Catalan Jacobians in Table \ref{table:momemntsa1} of Section \ref{sec:moments}.

In Section \ref{sec:GaloisType}, we study the Galois endomorphism types of Catalan Jacobians. There is a nice correspondence between the Sato-Tate group and the real endomorphism algebra for abelian varieties of dimension $g\leq 3$ (see Theorem 1.4 of \cite{FKS2012}). This may not always extend to higher dimension, but it does for Catalan Jacobians and other nondegenerate abelian varieties. We  approach this problem in two ways: through data obtained from moment statistics and by working with Rosati forms. The latter is the more traditional method of determining real endomorphism algebras, whereas the former is a new technique that uses recent results of Costa, Fit\'e, and Sutherland \cite{Costa2019}.

\subsection*{Notation and conventions} We begin by fixing notation used in later sections.  Let $A$ be an abelian variety defined over a number field $F$. The ring of endomorphisms of $A/F$ is denoted by $\End(A_F)$, or simply $\End(A)$ if the field of definition is clear from the context. 

The curve $y^q=x^p-1$ is denoted by $C_{p,q}$,  and we assume throughout the paper that $p$ and $q$ are distinct odd primes. We will write $\zeta_m$ for a primitive $m^{th}$ root of unity. For any rational number $x$ whose denominator is coprime to an integer $r$, $\langle x \rangle_r$ denotes the unique representative of $x$ modulo $r$ between 0 and $r-1$. 

Let $I$ denote the $2\times 2$ identity matrix and define the matrices
\begin{align}\label{eqn:Jmatrix}
    J=\begin{pmatrix}0&1\\-1&0\end{pmatrix}
\end{align}
and
\begin{align}\label{eqn:Zmatrix}
    Z=Z_{pq}=\diag(\zeta_{pq},\overline{\zeta_{pq}})
\end{align}
The symplectic form considered throughout the paper is given by
\begin{align}\label{eqn:symplecticH}
    H=\diag(\underbrace{J,\dots,J}_g).
\end{align}
Lastly, for any positive integer $n$, we define the following subgroup of the unitary symplectic group $\USp(2n)$
\begin{equation*}\label{eqn:U1n}
    \U(1)^n:=\left\langle \diag( u_1,\overline{u_1},\ldots, u_n,\overline{u_n}):u_i\in \mathbb C^\times, |u_i|=1\right\rangle.
\end{equation*}

%%%%%%%%%%%%%%%%%%%%%%%%%%%%%%%%%%%%%%%%%%%%%%%%%%%%%%%%%
\section{Background}\label{sec:background}

\subsection{An $\ell$-adic construction of the Sato-Tate group}
\label{sec:ASTgroupdefinition}
We begin by defining both the algebraic Sato-Tate group and the Sato-Tate group. We follow the exposition of \cite{EmoryGoodson2020} and \cite[Section 3.2]{SutherlandNotes}. See also \cite[Chapter 8]{SerreNXP}.

Let $A/ F$ be an algebraic variety of dimension $g$ defined over the number field $F$. We define the Tate module $T_{\ell}:=\varprojlim_{n} A[\ell^n]$, where $\ell$ is prime, a free $\mathbb{Z}_{\ell}$-module of rank $2g$,
and the rational Tate module $V_{\ell}:=T_{\ell}\otimes_{\mathbb{Z}} \mathbb{Q}$, a $\mathbb{Q}_{\ell}$-vector space of dimension $2g.$ The Galois action on the Tate module is given by an $\ell$-adic representation 
\begin{align}\label{eqn:artinrepresntation}
    \rho_{A,\ell}:\Gal(\overline{F}/F) \rightarrow \Aut(V_{\ell}) \cong \GL_{2g}(\mathbb{Q}_{\ell}).
\end{align}
The $\ell$-adic monodromy group of $A$, denoted $G_{A,\ell}$, is the Zariski closure of the image of this map in $\GL_{2g}(\mathbb{Q}_{\ell})$, and we define $G^{1}_{A,\ell}:=G_{A,\ell}\cap \Sp_{2g}(\mathbb{Q}_{\ell})$. Banaszak and Kedlaya proposed in \cite[Conjecture 2.1]{Banaszak2015} the following conjecture that was partly intended to be a refinement of the Mumford-Tate conjecture.

\begin{conjecture}[Algebraic Sato-Tate Conjecture]\label{conjec:AST}  There is an algebraic subgroup $\AST(A)$ of $\Sp_{2g}$ over $\mathbb Q$, called the \textbf{algebraic Sato-Tate group of $A$}, such that $\AST^0(A)$ is reductive and, for each prime $\ell$, $G^{1}_{A,\ell}=\AST(A)\otimes_{\mathbb Q} \mathbb Q_{\ell}$.
\end{conjecture}
When this conjecture holds we can define the \textbf{Sato-Tate group} of $A$, denoted   $\ST(A)$, to be a maximal compact Lie subgroup of $G^{1}_{A,\ell}\otimes_{\mathbb{Q}_{\ell}}\mathbb{C}=\AST(A)\otimes_\mathbb{Q} \mathbb{C}$ contained in $\USp(2g)$. It is conjectured that $\ST(A)$ is, up to conjugacy in $\USp(2g)$, independent of the choice of the prime $\ell$ and of the embedding of $\mathbb{Q}_\ell$ in $\mathbb{C}$ and so we will refer to $\ST(A)$ as {\it the} Sato-Tate group of $A$ (see, for example, \cite{SutherlandNotes}). While the Sato-Tate group is a compact Lie group, it may not be connected \cite{FS2016}. We denote  the connected component of the  identity  (also called the identity component) of $\ST(A)$ by  $\ST^0(A)$.

\subsection{Nondegeneracy}\label{sec:nondegeneracy}

Let $A$ be a nonsingular projective variety over $\mathbb C$. We denote the (complexified) Hodge ring of $A$ by
$$\mathscr B^*(A):=\displaystyle\sum_{d=0}^{\dim(A)} \mathscr B^d(A), $$
where $\mathscr B^d(A)=(H^{2d}(A,\mathbb Q)\cap H^{d,d}(A))\otimes \mathbb C$ is the $\mathbb C$-span of Hodge cycles of codimension $d$ on $A$.  
Furthermore, we define the ring $$\mathscr D^*(A):=\displaystyle\sum_{d=0}^{\dim(A)} \mathscr D^d(A)$$ 
where $\mathscr D^d(A)$ is the $\mathbb C$-span of classes of intersection of $d$-divisors. This is the subring of $\mathscr B^*(A)$ generated by the divisor classes, i.e.  generated by $\mathscr B^1(A)$. In general, it is known that we have containment $\mathscr D^*(A) \subseteq \mathscr B^*(A)$, and we say that an abelian variety $A$ is  \textbf{nondegenerate} if we have equality, i.e.  $\mathscr D^*(A) = \mathscr B^*(A)$ (see \cite{Aoki2002}). Furthermore, an abelian variety $A$ is said to be \textbf{stably nondegenerate} if, for any integer $k\geq 1$, $\mathscr D^*(A^k) = \mathscr B^*(A^k)$ \cite{Aoki2002,Hazama89}.  If $\mathscr B^*(A)$ is not generated by the divisor classes $\mathscr B^1(A)$, i.e.  $\mathscr D^*(A) \not= \mathscr B^*(A)$, then $A$ is said to be \textbf{degenerate}. The additional Hodge cycles that are not generated by divisor classes are referred to as \textbf{exceptional cycles} or \textbf{exceptional classes}  (see, for example, \cite{Murty1984}).

While nondegeneracy extends to powers of stably nondegenerate abelian varieties, it does not necessarily apply to products of nonisogenous varieties. For example, Shioda demonstrates in \cite{Shioda82} that the Jacobian of the genus 4  curve $y^2=x^9-1$ is degenerate, though it is the product of two simple, nondegenerate varieties. Aoki proves in \cite{Aoki2004} that generalized Catalan curves $C_{p^\mu,q^\nu}:y^{q^\nu}=x^{p^\mu}-1$, where $(\mu,\nu)\not=(1,1)$, also exhibit this phenomenon: each of the factors of the Jacobian are nondegenerate but $\Jac(C_{p^\mu,q^\nu})$ is itself degenerate. Hazama explains in \cite{Hazama89} that this can occur when one of the simple factors is of type-IV in the Albert's classification.

Hazama proves in Theorem 1.2 of \cite{Hazama89} that $A$ is stably nondegenerate if and only if the dimension of its Hodge group is maximal\footnote{In \cite{FiteGonzalezLario2016}, the authors use the word \emph{nondegenerate} to describe abelian varieties with this property.}. When $A$ is an abelian variety with  CM and is absolutely simple, this is equivalent to saying that the CM type is nondegenerate (see, for example, \cite{Aoki2002, FiteGonzalezLario2016}). The following result is proved in \cite{Hazama97}.

\begin{proposition}\label{prop:CatalanNondegnerate}
The Jacobian variety $\Jac(C_{p,q})$ of the Catalan curve is absolutely simple and  nondegenerate.
\end{proposition}
Hazama proves the nondegeneracy by proving that the CM-type is nondegenerate. This is done by verifying the non-vanishing of certain character sums attached to the Jacobian.

Nondegeneracy is related to the Hodge conjecture. Let $\mathscr C^d(A)$ be the subspace of $\mathscr B^d(A)$ generated by the classes of algebraic cycles on $A$ of codimension $d$. Then 
$$\mathscr D^d(A) \subseteq \mathscr C^d(A) \subseteq \mathscr B^d(A)$$
and the \textbf{Hodge conjecture} for $A$ asserts that $\mathscr C^d(A) = \mathscr B^d(A)$ for all $d$ \cite{Aoki2002, Shioda82}. It is clear from the definition that if $A$ is nondegenerate then the Hodge conjecture holds, and  if the Hodge conjecture does not hold then $A$ must be degenerate. However, there are many cases where the Hodge conjecture holds for degenerate abelian varieties (see, for example, \cite{Pohlmann1968,Shioda82}).

\subsection{The Twisted Lefschetz group}\label{sec:Lefschetz}

Banaszak and Kedlaya introduced in \cite{Banaszak2015} the \textbf{twisted Lefschetz  group}, denoted $\TL(A)$, which is a closed algebraic subgroup of  $\Sp_{2g}$ defined by
\begin{align}\label{eqn:TL}
\TL(A):=\bigcup_{\tau \in \Gal(\overline{F}/F)} \LL(A)(\tau),
\end{align}
where $\LL(A)(\tau):=\{\gamma \in \Sp_{2g}\mid \gamma \alpha \gamma^{-1}=\tau(\alpha) \text{ for all }\alpha \in \End(A_{\overline{F}})_\mathbb{Q}\}.$ 
When $\tau$ is the identity automorphism, $\LL(A)(\tau)$ forms a group, called the Lefschetz group, which is denoted in the literature as simply $\LL(A)$.

Banaszak and Kedlaya prove in \cite[Theorem 6.1]{Banaszak2015} that if the Mumford-Tate conjecture is explained by endomorphisms and the twisted Lefschetz group of $A/\overline{F}$ is connected, then the algebraic Sato-Tate conjecture holds and $\AST(A) = \TL(A).$ In this case, we say that the algebraic Sato-Tate conjecture is also  explained by endomorphisms. It follows from \cite[Theorem 6.1]{Banaszak2015} and work of Serre in \cite[Section 8.3.4]{SerreNXP} that when the algebraic Sato-Tate conjecture is  explained by endomorphisms the component group $\ST(A)/\ST^0(A)$ is isomorphic to the Galois group $\Gal(K/F)$, where $K$ is the \textbf{endomorphism field} of $A$, i.e. the minimal extension over which all the endomorphisms of the abelian variety $A$ are defined (see, for example, \cite[Theorem 3.12]{SutherlandNotes} and \cite[Proposition 2.17]{FKS2012}).  We can use the twisted Lefschetz group to write down explicit generators of the component group of the Sato-Tate group.

The Jacobian varieties we study in this paper satisfy these conditions, and so we have the following result.

\begin{proposition}\label{prop:AST=TL}
The algebraic Sato-Tate Conjecture holds for the Catalan curve $\Jac(C_{p,q})$ with $\AST(\Jac(C_{p,q}))=\TL(\Jac(C_{p,q}))$.
\end{proposition}

\begin{proof}
In Proposition \ref{prop:CatalanNondegnerate} we saw that $\Jac(C_{p,q})$ is a nondegenerate CM abelian variety. The result then  follows from Theorems 6.1 and 6.6 in \cite{Banaszak2015}.
\end{proof}

\subsection{The Sato-Tate axioms}\label{sec:SatoTateAxioms}
We conclude the background section by recalling some necessary conditions for a subgroup of $\USp(2g)$ to occur as a Sato-Tate group. These axioms will not be used in our work, but they have been used to classify all possible Sato-Tate groups for a given dimension $g\leq 3$ (see, for example, \cite{FKS2012, FiteKS_genus3_2019,FKS2021satotate}). We state the axioms as they are  laid out in \cite{FKS2021satotate}.  For a group $G$ with identity component $G^0$, the Sato-Tate group axioms in dimension $g$ are as follows.

\begin{enumerate}
    \item[(ST1)] The group $G$ is a closed subgroup of $\USp(2g)$.
    \item[(ST2)] (Hodge condition) There exists a homomorphism $\theta\colon \U(1)\to G^0$ such that $\theta(u)$ has eigenvalues $u,u^{-1}$ each with multiplicity $g$. The image of such a $\theta$ is called a {\it Hodge circle}, and the set of all Hodge circles generates a dense subgroup of $G^0$.
    \item[(ST3)] (Rationality condition) For each component $H$ of $G$ and irreducible character $\chi$ of $\GL(\mathbb{C}^{2g})$, the expected value (under the Haar measure) of $\chi(\gamma)$ over $\gamma\in H$ is an integer. 
    \item[(ST4)] (Lefschetz condition) The subgroup of $\USp(2g)$ fixing $\End(\mathbb{C}^{2g})^{G^0}$ is equal to $G^0$. 
\end{enumerate}

Proposition 3.2 of \cite{FKS2012} proves that, for any abelian variety satisfying both the Mumford-Tate conjecture and the algebraic Sato-Tate conjecture, the Sato-Tate group $G=\ST(A)$ satisfies the first three axioms. The last axiom is satisfied if, in addition, the Hodge group of $A$ equals the Lefschetz group of $A$ (see Proposition 3.3 of \cite{FKS2021satotate}). Catalan Jacobians  satisfy these properties, and so their Sato-Tate groups  satisfy the four axioms. 

\section{Endomorphisms of Catalan Jacobians} \label{sec:EndoGaloisAction}

Let $p\not=q$ be odd primes. The Jacobian of $C_{p,q}$ has complex multiplication by the field $K=\mathbb Q(\zeta_{pq})$, and so the Galois group $\Gal(K/\mathbb Q)$ is isomorphic to a product of cyclic groups $\mathbb Z/p\mathbb Z^\times\times\mathbb Z/q\mathbb Z^\times$. The main goals of this section are to give an explicit model of a generator for the group of endomorphisms  of $\Jac(C_{p,q})$ and to describe how the Galois group acts on the endomorphism.  We begin by defining some necessary notation that will be used when defining both the endomorphism and the Sato-Tate group of $\Jac(C_{p,q})$. 

\subsection{Galois elements and some notation}

Let $c$ be a generator for $\mathbb Z/p\mathbb Z^\times$ and $d$ be a generator for $\mathbb Z/q\mathbb Z^\times$. These elements yield the following generators of $\Gal(K/\mathbb Q)$:
\begin{center}
\begin{multicols}{2}
$\sigma_p\colon\begin{cases}
    \zeta_{p}\mapsto  \zeta_{p}^c\\
    \zeta_{q}\mapsto  \zeta_{q}
\end{cases}$

$\sigma_q\colon\begin{cases}
    \zeta_{p}\mapsto  \zeta_{p}\\
    \zeta_{q}\mapsto  \zeta_{q}^d.
\end{cases}$
\end{multicols}
\end{center}
In this paper, we will be working with powers of $\zeta_{pq}$ of the form $\zeta_{pq}^{q(a+1)-pb}.$ We  compute the actions of the Galois elements on these powers of $\zeta_{pq}$ to be
\begin{align}\label{eqn:GaloisActions}
    \sigma_p\left(\zeta_{pq}^{q(a+1)-pb}\right)&=\zeta_{pq}^{qc(a+1)-pb}& \text{and}&&\sigma_q\left(\zeta_{pq}^{q(a+1)-pb}\right)&=\zeta_{pq}^{q(a+1)-pdb}.
\end{align}

We now define some additional notation. Let $k_0=0$ and for any $b$ satisfying $1\leq b\leq q-1$, define
\begin{align}\label{eqn:kbdefinition}
    k_b=\left\lfloor\frac{pb-q-1}{q}\right\rfloor.
\end{align} 

The $k_b$-values are increasing with respect to $b$: $k_{b_1}\leq k_{b_2}$ whenever ${b_1}\leq {b_2}$. When $p<q$, the value $k_b$ may be negative for some values of $b$, and we let $\kappa_t$ be the sum
\begin{align}\label{eqn:kprimedefinition}
    \kappa_t=\sum_{\substack{b=0\\ k_b\geq 0}}^t(k_b+1),
\end{align} 
where we restrict to adding only the nonnegative values $k_b.$ 
% \begin{remark}
% The quantities in Equations \eqref{eqn:kbdefinition} and \eqref{eqn:kprimedefinition} are relevant to our work in the following way. The block entries of the endomorphism  of $\Jac(C_{p,q})$ will be defined in terms of combinations of values of $b$ and $a$ (defined below in Inequalities \eqref{eqn:brestrictions} and \eqref{eqn:arestrictions}, respectively). The number of $a$-values paired with a particular $b$-value will be $k_b+1$ whenever $k_b\geq 0$. The value $\kappa_b$ will help us keep track of the columns associated to particular $(a,b)$-pairings. Furthermore, $\kappa_{q-1}$ gives the total number of $(a,b)$-pairings, and so it equals the genus of $C_{p,q}.$
% \end{remark}

\subsection{Defining the endomorphism}

Let $\alpha\colon C_{p,q}\to C_{p,q}$ be the curve automorphism defined by $\alpha(x,y)=(\zeta_p x,\zeta_q y)=(\zeta_{pq}^q x,\zeta_{pq}^p y)$. This automorphism has order $pq$, and so it is a generator of the automorphism group of $ C_{p,q}$. It is then clear that the endomorphism ring of $\Jac(C_{p,q})$ satisfies the following.
\begin{proposition}\label{prop:endoringJac}
Let $K=\mathbb{Q}(\zeta_{pq})$. Then 
\begin{align*}
\End(\Jac(C_{p,q})_K)\simeq \mathbb{Z}[\zeta_{pq}],
\end{align*}
and for any intermediate field $L\subseteq K$, the endomorphism ring $\End(\Jac(C_{p,q})_L)$ is isomorphic to 
\begin{align*}
\mathcal O_L=\mathcal O_K^{\Gal(K/L)}\subseteq \mathcal{O}_K\simeq \End(\Jac(C_{p,q})_K).
\end{align*}
\end{proposition}
\begin{proof}
The second half of the statement follows from the fact that isomorphism in the first half of the statement is Galois-equivariant.
\end{proof}

Let $a$ and $b$ be  integers satisfying 
\begin{align}
    1\leq &b\leq q-1 \label{eqn:brestrictions}\\
    0\leq a<p&\;\;\text{ and }\;\;a\leq k_b. \label{eqn:arestrictions}
\end{align}

The genus of  $C_{p,q}$  is $g={(p-1)(q-1)}/{2}$ and we take as a basis for the space of regular $1$-forms $\Omega^1(C_{p,q})$ the set
\begin{align}\label{def:basisordering}
    B=\left\{\omega_{a,b}=\frac{x^adx}{y^b}\mid b\text{ satisfies Equation \eqref{eqn:brestrictions}} , a\text{ satisfies Equation \eqref{eqn:arestrictions}} \right\}.
\end{align}  

% In order to determine explicit generators for the Sato-Tate group, we must write an explicit expression for the endomorphism $\alpha$ of $\Jac(C_{p,q})$.  We start by choosing an ordering of the basis elements. Any choice that is made for the ordering of the basis elements of $\Omega^1(C_{p,q})$ will lead to an isomorphic Sato-Tate group of $\Jac(C_{p,q})$.
% \begin{definition} \label{def:basisordering}
% Start with $b=1$ and list all $\omega_{a,b}$, where the $a$ are increasing and satisfy the inequalities in \eqref{eqn:arestrictions}. Increment $b$ and repeat this process while Equation \eqref{eqn:brestrictions} is still satisfied. 
% \end{definition}
We compute pullbacks of the differentials with respect to $\alpha$ in order to determine the associated endomorphism $\alpha$ of $\Jac(C_{p,q})$:
$$\alpha^*\left(\omega_{a,b}\right)=\zeta_{pq}^{q(a+1)-pb}\omega_{a,b}.$$
 We can now define the induced endomorphism $\alpha$ of $H_1(\Jac(C_{p,q})_\mathbb{C},\mathbb Q)$.
\begin{definition}\label{def:endomorphismalpha}
By taking the symplectic basis of $H_1(\Jac(C_{p,q})_\mathbb{C},\mathbb C)$ (with respect to the matrix $H=\diag(J,\ldots,J)$) corresponding to $B$ in Equation \eqref{def:basisordering},  $\alpha$ is the diagonal matrix whose $i^{th}$ diagonal $2\times2$  block is
$$\alpha[i]=Z^{q(a_i+1)-pb_i},$$
where $a_i=i-\kappa_{b_i-1}-1$ and $b_i$ is the largest integer satisfying Equation \eqref{eqn:brestrictions}   such that $a_i>0$. 
\end{definition}

\begin{example}\label{ex:q3p7a}
Consider the genus 6 curve  $C_{7,3}\colon y^3=x^7-1$. Using the strategy described above, we write the basis of $\Omega^1(C_{p,q})$ as $\{\omega_{0,1},\;\omega_{1,1},\;\omega_{0,2},\;\omega_{1,2},\;\omega_{2,2},\;\omega_{3,2}\}$ and the endomorphism $\alpha$ as
$$\alpha=\diag(Z^{q-p},Z^{2q-p},Z^{q-2p},Z^{2q-2p},Z^{3q-2p},Z^{4q-2p}).$$
\end{example}

%%%%%%%%%%%%%%%%%%%%%%%%%%%%%%%%%%%%%%%%%%%%%%%%%%%%%%%%%%%%%%%%%%
\subsection{The actions of the Galois elements on $\alpha$}\label{sec:GaloisAction}
The main results of this subsection are Propositions \ref{prop:sigma_q} and \ref{prop:sigma_p}, where we describe the action of the Galois elements on the endomorphism $\alpha$. The following lemma will be referenced in the proofs of these propositions. 

\begin{lemma}\label{lemma:floorbounds}
Let $\beta$ be a positive integer satisfying Inequality \eqref{eqn:brestrictions}, and define two quantities $y_1=(p(q-\beta)-q+1)/q$ and $y_2=(p(q-\beta)-q-1)/q$. Then any integer $\lambda$  satisfying $\lambda<y_1$ also satisfies  $\lambda\leq\lfloor y_2\rfloor$.
\end{lemma}
\begin{proof}
The quantities $y_1$ and $y_2$ differ by $2/q$: $y_2=y_1-2/q$. Let $\lambda$ be any integer less than $y_1$. We split the proof of the lemma into two cases: $y_1\in \mathbb{Z}$ and $y_1\not\in \mathbb{Z}$. 

First suppose that $y_1\in \mathbb Z$ so that $\lambda\leq y_1-1$. Since $0<2/q<1$, the inequalities $y_1-1 < y_1-2/q<y_1$ hold. Hence, $y_1-1=\lfloor  y_1-2/q\rfloor =\lfloor y_2\rfloor$, and so $\lambda\leq \lfloor y_2\rfloor$.

Suppose instead that $y_1\not\in \mathbb Z$. We claim that $y_1-1/q$ is also not an integer. Observe that
$$y_1-\frac1q = \frac{p(q-\beta)-q}{q}=p-1-\frac{p\beta}{q}.$$
Since both $p$ and $\beta$ are relatively prime to $q$, the fractional term $\frac{p\beta}{q}$ and, hence, $y_1-1/q$  are not integers. This implies that $y_1>\lfloor y_1-1/q\rfloor =\lfloor y_1-2/q\rfloor =\lfloor y_2 \rfloor$. Combining this with the requirement that $\lambda$ is an integer less than $y_1$ yields the desired result of $\lambda\leq \lfloor y_2\rfloor$.
\end{proof}

We will now give a complete description of the action of $\sigma_q$ on the endomorphism $\alpha$. We will do this by studying the action on an arbitrary diagonal block entry of the matrix. Using the action described in Equation \eqref{eqn:GaloisActions}, we see that
$${}^{\sigma_q}\alpha[i] = Z^{q(a_i+1)-pdb_i}.$$ 
A particularly nice property of this action is that ${}^{\sigma_q}\alpha[i]$ is either equal to or is the conjugate of some block diagonal entry of the matrix $\alpha$ (we denote the conjugate of $\alpha[j]$ by $\overline{\alpha[j]}$).

\begin{proposition}\label{prop:sigma_q}
Let $d$ be a generator of the cyclic group $\mathbb{Z}/q\mathbb{Z}^\times$ and let $t_i=\langle db_i \rangle_q$. Define the quantities $\iota_1=\kappa_{t_i-1}+a_i+1$ and $\iota_2 =\kappa_{t_i'-1}+p-(a_i+1)$, where $t_i'=q-t_i$. Then ${}^{\sigma_q}\alpha$ is the diagonal matrix whose $i^{th}$ diagonal $2\times2$ block is
$${}^{\sigma_q}\alpha[i]=\begin{cases}
\alpha[\iota_1]&\text{if } 0\leq a_i\leq k_{t_i}, \\
\overline{\alpha[\iota_2]}&\text{if } a_i> k_{t_i},\\
0&\text{otherwise.}  
\end{cases}$$
\end{proposition}

\begin{proof}
First suppose that $a_i$ satisfies $0\leq a_i\leq k_{t_i}$.
Then $t_i$ and $a_i$ together satisfy the bounds of Equations \eqref{eqn:brestrictions} and \eqref{eqn:arestrictions}. Hence, ${}^{\sigma_q}\alpha[i] = Z^{q(a_i+1)-pt_i} = \alpha[\iota],$
for some $\iota$. Comparing this to the definition of $\kappa$ and doing a bit of arithmetic shows that  $\iota=\iota_1$.

Suppose instead that $a_i$ satisfies
$$a_i\geq k_{t_i}+1>\frac{pt_i-q-1}{q}. $$
Note that this includes the case where $k_{t_i}<0$. Then  $t_i$ and $a_i$ do not satisfy the bounds of Equations \eqref{eqn:brestrictions} and \eqref{eqn:arestrictions}. Let $a_i'=p-a_i-2$ and $t_i'=q-t_i$. Then 
$$Z^{q(a_i+1)-pt_i} = \overline{Z^{q(a_i'+1)-pt_i'}}.$$
The bounds on $t_i$ imply that $t_i'$ satisfies $1\leq t_i'\leq q-1$ and the bounds on $a_i$ imply 
\begin{align*}
    a_i' &<p-\frac{pt_i-q-1}{q}-2\\
            &=\frac{p(q-t_i)-q+1}{q}\\
            &=\frac{pt_i'-q+1}{q}.
\end{align*}
By Lemma \ref{lemma:floorbounds}, we can conclude that  $a_i'\leq \lfloor(pt_i'-q-1)/{q} \rfloor=k_{t_i'}$. Thus, $a_i'$ and $t_i'$ together satisfy the bounds of  Inequalities \eqref{eqn:brestrictions} and \eqref{eqn:arestrictions}. Hence, $Z^{q(a_i'+1)-pt_i'}$ appears as a block entry of $\alpha$ and
$${}^{\sigma_q}\alpha[i] = \overline{Z^{q(a_i'+1)-pt_i'}}=\overline{ \alpha[\iota]}$$ 
for some $\iota$. Comparing this to the definition of $\kappa$  and working out a bit of arithmetic shows that  $\iota=\iota_2$.

\end{proof}

%%%%%%%%%%%%%%%%%%%%%%%%%%%%%%%%%%%%%%%%%%%%%%%%%%%%%%%%%%%%%%%%%%

We can use a similar strategy to describe of the action of $\sigma_p$ on the endomorphism $\alpha$.  Using the action described in Equation \eqref{eqn:GaloisActions}, we see that
$${}^{\sigma_p}\alpha[i] = Z^{qc(a_i+1)-pb_i}.$$ This leads to the following proposition.  

\begin{proposition}\label{prop:sigma_p}
Let $c$ be a generator of the cyclic group $\mathbb{Z}/p\mathbb{Z}^\times$ and let $s_i=\langle c(a_i+1)\rangle_p-1$. Let $\iota_1=\kappa_{b_i-1}+s_i+1$ and $\iota_2 =\kappa_{b_i'-1}+p-(s_i+1)$, where $b_i'=q-b_i$. Then ${}^{\sigma_q}\alpha$ is the diagonal matrix whose $i^{th}$ diagonal $2\times2$ block is
$${}^{\sigma_p}\alpha[i]=\begin{cases}
\alpha[\iota_1]&\text{if } 0\leq s_i\leq k_{b_i}, \\
\overline{\alpha[\iota_2]}&\text{if } s_i> k_{b_i},\\
0&\text{otherwise.}  
\end{cases}$$
\end{proposition}

\begin{proof}

The bounds for $s_i$ are $0\leq s_i<p-1$, and $k_{b_i}\geq0$ since it is associated to a pair $a_i, b_i$ that appear in an entry of the endomorphism $\alpha$.  We follow the method of proof for Proposition \ref{prop:sigma_q} and split into two cases. 

First suppose that $s_i$ satisfies $0\leq s_i\leq k_{b_i}. $
Then $b_i$ and $s_i$ together satisfy the bounds of Equations \eqref{eqn:brestrictions} and \eqref{eqn:arestrictions}. Hence, ${}^{\sigma_p}\alpha[i] = Z^{q(s_i+1)-pb_i} = \alpha[\iota],$
for $\iota=\iota_1$.

Suppose instead that $s_i$ satisfies
$$s_i\geq k_{b_i}+1>\frac{pb_i-q-1}{q}. $$
Let $s_i'=p-s_i-2$ and $b_i'=q-b_i$. Then  $Z^{q(s_i+1)-pb_i} = \overline{Z^{q(s_i'+1)-pb_i'}}.$
The bounds on $s_i$ imply that
\begin{align*}
    s_i' &<p-\frac{pb_i-q-1}{q}-2\\
            &=\frac{p(q-b_i)-q+1}{q}\\
            &=\frac{pb_i'-q+1}{q}.
\end{align*}
By Lemma \ref{lemma:floorbounds}, this proves that $s_i'\leq \lfloor(pb_i'-q-1)/{q} \rfloor=k_{b_i'}$. Thus, $b_i'$ and $s_i'$ together satisfy the bounds of  Equations \eqref{eqn:brestrictions} and \eqref{eqn:arestrictions}. Hence, $Z^{q(s_i'+1)-pb_i'}$ appears as a block entry of $\alpha$ and
$${}^{\sigma_q}\alpha[i] = \overline{Z^{q(s_i'+1)-pb_i'}}$$ 
which is the conjugate of the block $\alpha[\iota]$,
for $\iota=\kappa_{b_i'-1}+s_i'+1= \kappa_{b_i'-1}+p-s_i-1=\iota_2$.

\end{proof}

\begin{example}\label{ex:q3p7b}
We return to the curve $C_{7,3}$ from Example \ref{ex:q3p7a}. The action of $\sigma_q$ (using $d=2$) on $\alpha$ yields the diagonal matrix
$${}^{\sigma_q}\alpha=\diag(Z^{q-2p},Z^{2q-2p},Z^{q-p},Z^{2q-p},\overline{Z^{4q-2p}},\overline{Z^{3q-2p}}).$$
The action of $\sigma_p$ (using $c=3$) on $\alpha$ yields 
$${}^{\sigma_p}\alpha=\diag(\overline{Z^{4q-2p}},\overline{Z^{q-2p}},{Z^{3q-2p}},\overline{Z^{q-p}},{Z^{2q-2p}},\overline{Z^{2q-p}}).$$
\end{example}

%%%%%%%%%%%%%%%%%%%%%%%%%%%%%%%%%%%%%%%%%%%%%%%%%%%%%%%%%%%%%%%%%%
\section{The Sato-Tate Group of $\Jac(C_{p,q})$}\label{sec:SatoTateGroup} 

We begin with a definition and a lemma that we need in order to define the component group and to prove Theorem \ref{thm:STgroupqsmaller}.

\subsection{Preliminaries}

\begin{definition}\label{def:blocksignedperm}
A \textbf{signed permutation matrix} is a square matrix that has exactly one entry of 1 or $-1$ in each row and each column and 0s everywhere else. We will define a \textbf{block signed permutation matrix} to be an even dimension signed permutation matrix that is partitioned into $2\times2$ blocks that are either $I$, $J$, or the zero matrix.
\end{definition}

For any matrix partitioned $X$ into $2\times2$ blocks, we will let $X[i,j]$ be the block in the $i^{th}$ row partition and $j^{th}$ column partition.

The general strategy for writing down a matrix in the component group of the Sato-Tate group that is associated to a Galois element $\sigma \in \Gal(K/\mathbb Q)$  is described in the following lemma. 

\begin{lemma}\label{lemma:BuildGamma}
Let $A$ be a diagonal matrix partitioned into $2\times2$ blocks and let $B$ be a block signed permutation matrix. Then $BAB^{-1}$ is a diagonal matrix. Furthermore, if $B[i,j]=I$ then $BAB^{-1}[i,i]=A[j,j]$ and if $B[i,j]=J$ then $BAB^{-1}[i,i]=-JA[j,j]J$.
\end{lemma}

\begin{proof}
The proof requires little more than elementary linear algebra. The first statement, that $BAB^{-1}$ is a diagonal matrix, comes from the fact that $A$ is diagonal and $B$ is a block signed permutation matrix.

Now note that the block entry $B^{-1}[j,i]$ is the inverse of the block $B[i,j]$. The $i^{th}$ diagonal block entry of $BAB^{-1}$ is the product
$$BAB^{-1}[i,i]= B[i,j]\cdot A[j,j] \cdot B^{-1}[j,i].$$
Thus, if $B[i,j]=I$ then the $i^{th}$ entry of $BAB^{-1}$ is $A[j,j]$. If, instead, $B[i,j]=J$ then $BAB^{-1}[i,i]=JA[j,j]J^{-1}=-JA[j,j]J.$
\end{proof}

\begin{remark}
As noted in the Introduction, this lemma can be applied to more varieties than those we study in this paper. For example, this lemma could be used for the varieties studied in \cite{EmoryGoodson2020, FiteGonzalezLario2016, LarioSomoza2018} and \cite{FS2016} (for certain values of $c$).
\end{remark}

\subsection{The identity component and the component group}\label{subsec:satotate}

We begin this subsection with a result about the identity component of the Sato-Tate group. 

\begin{proposition}\label{prop:idcomponent} 
The identity component of the Sato-Group of $\Jac(C_{p,q})$ is
$$\ST^0(\Jac(C_{p,q}))\simeq \U(1)^{g},$$
where $g=(p-1)(q-1)/2$ is the genus of $C_{p,q}$.
\end{proposition}

\begin{proof}
Recall the definition of the algebraic Sato-Tate group $\AST(\Jac(C_{p,q}))$ in Conjecture \ref{conjec:AST}. If we consider the same construction but restrict the domain of the $\ell$-adic representation $\rho_{A,\ell}$ to $\Gal(\overline{ \mathbb Q}/K)$, where $K$ is the CM field of $\Jac(C_{p,q})$, we obtain the identity component of $\AST(\Jac(C_{p,q}))$. The identity component of the Sato-Tate group of $\Jac(C_{p,q})$ is a maximal compact subgroup of this.

This restricted map is exactly the map the appears in Theorem A of \cite{Banaszak2003}, and so we can apply that result  to our situation. The theorem states that the image of the restricted map is contained in the group of diagonal matrices of the form
$$\{\diag( x_1,y_1,\dots,x_{g},y_{g})\mid   x_i,y_i\in\mathbb{Q}_{\ell}^{\times}, x_1y_1=\cdots=x_{g}y_{g}=1\}.$$

The Jacobian of $C_{p,q}$ is absolutely simple and nondegenerate (see Proposition \ref{prop:CatalanNondegnerate}), and so, as in \cite{EmoryGoodson2020,FiteGonzalezLario2016}, we can go one step further and conclude that the containment is actually an equality. It follows that we can choose $\ST^0(\Jac(C_{p,q}))$ to be the maximal compact subgroup $\U(1)^g$.

\end{proof}

We now define two matrices. We will then prove that they form a generating set for the component group of the Sato-Tate group.  For the following definitions, let $g$ be the genus of $C_{p,q}$. 

\begin{definition}\label{def:gamma_q}
Let $d$ be a generator of the cyclic group $\mathbb{Z}/q\mathbb{Z}^\times$. For any integer $i$ satisfying $1\leq i \leq g$, let $t_i=\langle db_i \rangle_q$ and $r_i=i-\kappa_{b_i-1}-1$, where $b_i$ is the largest positive integer such that $r_i$ is nonnegative.  Define $\gamma_q$ to be the $2g\times2g$ block signed permutation matrix whose $ij^{th}$ block is
$$\gamma_q[i,j]=\begin{cases}
I&\text{if } j=\kappa_{t_i-1}+r_i+1 \text{ and } 0\leq r_i\leq k_{t_i}, \\
J&\text{if } j=\kappa_{t_i'-1}+p-(r_i+1), \text{ and } r_i> k_{t_i},\\
0&\text{otherwise,}  
\end{cases}$$
where $t_i'=q-t_i$.
\end{definition}

\begin{definition}\label{def:gamma_p}
Let $c$ be a generator of the cyclic group $\mathbb{Z}/p\mathbb{Z}^\times$, $r_i$ and $b_i$ be as defined in Definition \ref{def:gamma_q}, and $s_i=\langle cr_i \rangle_p-1$. Define $\gamma_p$ to be the  $2g\times2g$  block signed permutation matrix whose $ij^{th}$ block is 
$$\gamma_p[i,j]=\begin{cases}
I&\text{if } j=\kappa_{b_i-1}+s_i+1 \text{ and } 0\leq s_i\leq k_{b_i},\\ 
J&\text{if } j=\kappa_{b'_i-1}+p-(s_i+1) \text{ and } s_i> k_{b_i},\\ 
0&\text{otherwise,}  
\end{cases}$$
where $b_i'=q-b_i$.
\end{definition}

\begin{theorem}\label{thm:STgroupqsmaller}
Let $q\not=p$ be odd primes and $\gamma_q,\gamma_p$ be as defined in Definitions \ref{def:gamma_q} and \ref{def:gamma_p}.  The Sato-Tate group of the Jacobian of the genus $g$ curve $C_{p,q}$ is 
$$\ST(\Jac(C_{p,q}))\simeq \langle\\U(1)^g, \gamma_q,\gamma_p\rangle.$$
\end{theorem}

\begin{proof}
We proved in Proposition \ref{prop:idcomponent} that the identity component $\ST^0(\Jac(C_{p,q}))$ is $\U(1)^g$. We will now prove the claim regarding the component group.

It follows from Proposition \ref{prop:AST=TL} that we can use the twisted Lefschetz group of $\Jac(C_{p,q})$ to write down explicit generators of the component group. Choose generators $c$ and $d$ of $\mathbb{Z}/p\mathbb{Z}^\times$ and $\mathbb{Z}/q\mathbb{Z}^\times$, respectively.  Propositions \ref{prop:sigma_q} and \ref{prop:sigma_p} show that the actions of the Galois elements simply rearrange and sometimes conjugate the diagonal entries of the endomorphism $\alpha$. Since $-JZJ=\overline{Z}$, where $Z$ is defined in Equation \eqref{eqn:Zmatrix},  we can apply Lemma \ref{lemma:BuildGamma} to determine the component group generators.

From here, we simply compare Definition \ref{def:gamma_q} to Proposition \ref{prop:sigma_q} and Definition \ref{def:gamma_p} to Proposition \ref{prop:sigma_p}. For the former, $r_i=a_i$ and so the definition of $\gamma_q$ yields the correct rearranging and conjugating as determined by the action of $\sigma_q$. For the latter, we used the same notation and so it is clear that the definition of $\gamma_p$ yields the rearranging and conjugating determined by the action of $\sigma_p$. Thus, $\gamma_q$ and $\gamma_p$, when taken with the identity component, are generators of $\TL(\Jac(C_{p,q}))$.
\end{proof}

The description of $\gamma_q$ simplifies rather nicely in the case where $q=3$, and so we present the following corollary.

\begin{corollary}
Let $q=3$, $p>3$ be  prime, and $m=\lfloor \frac{p-4}{3}\rfloor+1$. Then the component group generator $\gamma_q$ is 

$$\gamma_q= \left(\begin{array}{c|c|c}
     &I_{2m}&  \\ 
     \hline
   I_{2m}  & &\\
   \hline
   &&J_{g-2m}
\end{array} \right),$$
where $I_{2m}$ is the $2m\times2m$ identity matrix and $J_{g-2m}=\antidiag(\underbrace{J,J,\ldots, J}_{g-2m})$. 
\end{corollary}

\begin{proof}
If $q=3$, then the genus is $g=p-1$ and the values of $b$ satisfying Equation \eqref{eqn:brestrictions} are $b_1=1$ and $b_2=2$. We choose $d=2$ as the generator of $\mathbb{Z}/q\mathbb{Z}^\times$, which leads to $t_1=2$ and $t_2=1$. Furthermore, $k_1=m-1$, $\kappa_1=m$, and $k_2=\lfloor (2p-4)/3\rfloor$.

The values of $i$ associated to $b_1$ are $1\leq i\leq m$ (so that $r_i\geq 0$). For these values of $i$, $0\leq r_i\leq$ is always satisfied. Hence, $\gamma_q[i,j]=I$ when 
$$j=\kappa_1 +r_i+1=m+i.$$
This verifies that the $I_{2m}$ in the top rows of $\gamma_q$ in the statement of the corollary are correct.

The values of $i$ associated to $b_2$ are then $m+1\leq i\leq g=p-1$. The associated values are $r_i$ are split into to intervals: $0\leq r_i\leq m-1$ and $r_i\geq m$. In the first interval we have $m+1\leq i\leq 2m$, which leads to $\gamma_q[i,j]=I$ when $j=i-m$. This verifies that the $I_{2m}$ in the middle rows of $\gamma_q$ in the statement of the corollary are correct.

In the second interval, $r_i\geq m$ and we need to determine which values of $j$ will lead to $\gamma_q[i,j]=J$. Let $i=2m+h$, where $1\leq h\leq g-2m$, so that $r_i=m+h-1$. Then the value of $j$ we need is
$$j=m+p-(m+h)=p-h= g-(h-1).$$
Thus, the $J$ blocks will start in row $i=2m+1$, column $j=g$, and will then cascade down and to the left. This verifies that the $J_{g-2m}$ in the bottom rows of $\gamma_q$ in the statement of the corollary are correct.
\end{proof}
\begin{remark}
See the Appendix for examples of the component group generators.
\end{remark}

\section{Moment Statistics}\label{sec:moments}
In this section, we describe the distributions of the coefficients of the characteristic polynomial of  random conjugacy classes in the Sato-Tate groups in Theorem \ref{thm:STgroupqsmaller}. For comparison, we compute the numerical $a_1$-moments of the normalized L-polynomial for some small genus Catalan curves.

\subsection{Preliminaries}\label{sec:MomentsIntro}
The following background information has been adapted from \cite[Section 6]{EmoryGoodson2020}. 

We start by recalling some basic properties of moment statistics. We define the $n^{th}$ moment (centered at 0) of a probability density function to be the expected value of the $n$th power of the values, i.e. $M_n[X]=E[X^n]$. Recall  that for independent variables $X$ and $Y$ we have $E[X+Y]=E[X]+E[Y]$ and $E[XY]=E[X]E[Y]$ (see, for example, \cite{LarioSomoza2018}). This yields the following identity %ies for moment statistics: $M_n[XY]=M_n[X]M_n[Y]$, $M_a[X]M_b[X]=M_{a+b}[X]$, and
\begin{align*}
    M_n[X_1+\cdots + X_m]=\sum_{\beta_1+\cdots+\beta_m=n} \binom{n}{\beta_1,\ldots, \beta_m}M_{\beta_1}[X_1]\cdots M_{\beta_m}[X_m].
\end{align*}
Furthermore, for any constant $b$, we have $M_n[b]=b^n$.

We start with the unitary group $\U(1)$ and consider the trace map  on $U\in \U(1)$ defined by $z:=\tr(U)=u+\overline{u}$. %=2\cos(\theta)$, where $u=e^{i\theta}$. 
This trace map takes values in $[-2,2]$. From here we see that %$dz=2\sin(\theta)d\theta$ and 
$$\mu_{\U(1)}= \frac1\pi \frac{dz}{\sqrt{4-z^2}}$$
gives a uniform measure of $\U(1)$ %on $\theta\in[-\pi,\pi]$ 
(see \cite[Section 2]{SutherlandNotes}). We can deduce the following pushforward measure
\begin{equation*}\label{eqn:muU1}
    \mu_{\U(1)^n}= \prod_{i=1}^n \frac1\pi \frac{dz_i}{\sqrt{4-z_i^2}}.
\end{equation*} 

We can now define the moment sequence $M[\mu]$, where $\mu$ is a positive measure on some interval $I=[-d,d]$. The $n^{th}$ moment $M_n[\mu]$  is, by definition,   $\mu(\phi_n)$, where $\phi_n$ is the function $z\mapsto z^n$. %It is therefore given by %$$M_n[\mu] = \int_I z^n\mu(z).$$
This yields $M_n[\mu_{\U(1)}]  = \binom{n}{n/2}$, where $\binom{n}{n/2}=0$ if $n$ is odd. Hence, $M[\mu_{\U(1)}] = (1,0,2, 0,6,0,20,0,\ldots).$ From here, we  take binomial convolutions to obtain
\begin{align}
    M_n[\mu_{\U(1)^g}] &=\sum_{\beta_1+\cdots+\beta_g=n} \binom{n}{\beta_1,\ldots, \beta_m}M_{\beta_1}[\mu_{\U(1)}]\cdots M_{\beta_g}[\mu_{\U(1)}].\label{eqn:MomentsU1g}
\end{align}

In what follows, for each $i\in\{1,2, \ldots, g\}$, denote by $\mu_i$ the projection of the Haar measure onto the interval $\left[-\binom{2g}{i},\binom{2g}{i}\right].$ We can compute $M_n[\mu_i]$ by averaging over the components of the Sato-Tate group. 

\subsection{Characteristic polynomials}\label{subsec:charpolys}

In this subsection, we give results for the characteristic polynomials in each component of the Sato-Tate group. Let $U$ be a random matrix in the identity component of $\ST(\Jac(C_{p,q}))$ and  let $A_{m,n}=\gamma_p^m\gamma_q^n$ be a matrix in the component group $\ST(\Jac(C_{p,q}))/\ST^0(\Jac(C_{p,q}))$. The characteristic polynomial of $U\cdot A_{m,n}$ is a degree $2g$ palindromic polynomial of the form
$$P_{m,n}(T)=T^{2g}+b_1T^{2g-1}+b_2T^{2g-2}+\cdots +b_2T^2+b_1T+1.$$ 

We now present two propositions regarding the coefficients of the characteristic polynomials.

\begin{proposition}\label{prop:charpolya1}
Only matrices in the identity component $\U(1)^g$  in $\ST(\Jac(C_{p,q})$ have a characteristic polynomial with  nonzero  $b_1$-coefficient.
\end{proposition}

\begin{proof}
The $b_1$-coefficient of the characteristic polynomial of $P_{m,n}$ is the trace of the matrix product $U\cdot A_{m,n}$. We will show that this trace is nonzero if and only if $m\equiv 0 \pmod{p-1}$ and $n\equiv 0 \pmod{q-1}$. 

To do this it is sufficient to show that $A_{m,n}$ has no nonzero diagonal entries. If  $A_{m,n}$ has a nonzero diagonal entry then, for some $j$, the $j^{th}$ diagonal block entry of $A_{m,n}$ is $I$ and $\alpha$ satisfies
\begin{equation}\label{eqn:powersofactions}
    {}^{(\sigma_p)^m\circ(\sigma_q)^n}\alpha[j]=\alpha[j].
\end{equation}
To determine if and when this is possible, we consider the action of the corresponding Galois element $(\sigma_p)^m\circ(\sigma_q)^n$. Let $\alpha[j]=Z^{q(a+1)-pb}$. Equation \eqref{eqn:powersofactions} holds if and only if  the following is true:
$$\zeta_{pq}^{q(a+1)-pb}={}^{(\sigma_p)^m\circ(\sigma_q)^n}\left(\zeta_{pq}^{q(a+1)-pb}\right)=\zeta_{pq}^{qc^m(a+1)-pd^nb}.$$
This equality implies that $$\zeta_{p}^{(a+1)}\overline\zeta_{q}^{b}=\zeta_{p}^{c^m(a+1)}\overline\zeta_{q}^{d^nb}.$$
Since $p$ and $q$ are relatively prime, we can conclude from here that 
$$\zeta_{p}^{(a+1)}=\zeta_{p}^{c^m(a+1)}={}^{(\sigma_p)^m}\zeta_{p}^{(a+1)}\;\;\text{ and }\;\; \zeta_{q}^{b}=\zeta_{q}^{d^nb}={}^{(\sigma_q)^n}\zeta_{q}^{b}.$$
Since $\sigma_p$ and $\sigma_q$ are generators of $\Gal(K/\mathbb{Q})$ and have orders $p-1$ and $q-1$, respectively, the above equalities hold if and only if $m\equiv 0 \pmod{p-1}$ and $n\equiv 0 \pmod{q-1}$. 
\end{proof}

One benefit of this result is that, in order to compute the $\mu_1$-moment statistics of the Sato-Tate group, we need only compute moments for a random matrix $U$ in the identity component of $\ST(\Jac(C_{p,q}))$. To average over the components we simply divide this value by the size of the Galois group.

\begin{proposition}\label{prop:pcharpolygammag}
Let $m={(p-1)/2}$ and $n={(q-1)/2}$. Then the characteristic polynomial of $U\cdot A_{m,n}$ in the Sato-Tate group is $$P_{m,n}(T)=(T^2+1)^g$$
for any $U$ in the identity component.
\end{proposition}
\begin{proof}
We will show that the matrix $A_{m,n}$ for these values of $m$ and $n$ has $J$ or $-J$ as each of its $g$ diagonal blocks. Since the characteristic polynomial of $\pm J$ is $T^2+1$, this will yield the desired result.

The actions and orders of the Galois elements $\sigma_p$ and $\sigma_q$ tell us that $\sigma_p^{m}(\zeta_p)=\overline \zeta_p $ and $\sigma_q^{n}(\zeta_q)=\overline \zeta_q $. Thus, 
\begin{align*}
    \sigma_p^{m}\circ\sigma_q^{n}(\zeta_{pq}^{q(a+1)-pb})&=\sigma_p^{m}\circ\sigma_q^{n}(\zeta_{p}^{(a+1)}\overline\zeta_{q}^{b})\\
    &=\overline\zeta_{p}^{(a+1)}\zeta_{q}^{b}\\
    &=\overline{\zeta_{pq}^{q(a+1)-pb}}.
\end{align*}
Thus, every diagonal entry of $\alpha$ satisfies
$${}^{\sigma_p^{m}\circ\sigma_q^{n}}\alpha[j]=\overline{\alpha[j]},$$
which means that the block diagonal entries of the corresponding matrix are all either $J$ or $-J$.
\end{proof}

\begin{example}
We determine moment statistics for the genus 4 curve $C_{5,3}:\; y^3=x^{5}-1$ over each subfield of the CM field $K=\mathbb{Q}(\zeta_{15})$. Using characteristic polynomials, we can compute the $n$th moments for each $\mu_i$, $1\leq i\leq4$. There is a table  corresponding to each of the four subfields: $\mathbb{Q}$ (Table \ref{table:moments35Q}), $\mathbb{Q}(\zeta_3)$ (Table \ref{table:moments35Q3}), $\mathbb{Q}(\zeta_5)$ (Table \ref{table:moments35Q5}), and $\mathbb{Q}(\zeta_{15})$ (Table \ref{table:moments35Q15}).  These moments were computed using Sage \cite{Sage}.  
\begin{table}[h]
{\renewcommand{\arraystretch}{1.5}
\begin{tabular}{|c|l|}
\hline
$M[\mu_1]$& $(1, 0, 1, 0, 21, 0, 640, 0, 23765, \ldots)$\\
\hline
$M[\mu_2]$&$(1, 1, 8, 76, 1168, 20956, 414284, 8643328, 187416464, \ldots)$\\
\hline
$M[\mu_3]$&$(1, 0, 13, 0, 11745, 0, 17177080, 0, 31036079585, \ldots)$\\
\hline
$M[\mu_4]$&$(1, 2, 27, 476, 18391, 689812, 34599990, 1677458008, 91894386279,  \ldots)$\\
\hline
\end{tabular}}
\caption{Moment Statistics for $y^3=x^{5}-1$ over $\mathbb{Q}$.}\label{table:moments35Q}
\end{table}

\begin{table}[h]
{\renewcommand{\arraystretch}{1.5}
\begin{tabular}{|c|l|}
\hline
$M[\mu_1]$& $(1, 0, 2, 0, 42, 0, 1280, 0, 47530, \ldots)$\\
\hline
$M[\mu_2]$&$(1, 1, 11, 136, 2263, 41656, 827444, 17282560, 374815319, \ldots)$\\
\hline
$M[\mu_3]$&$(1, 0, 26, 0, 23490, 0, 34354160, 0, 62072159170, \ldots)$\\
\hline
$M[\mu_4]$&$(1, 2, 42, 890, 36418, 1377502, 69187410, 3354841408, 183788328258,  \ldots)$\\
\hline
\end{tabular}}
\caption{Moment Statistics for $y^3=x^{5}-1$ over $\mathbb{Q}(\zeta_3)$.}\label{table:moments35Q3}
\end{table}

\begin{table}[h]
{\renewcommand{\arraystretch}{1.5}
\begin{tabular}{|c|l|}
\hline
$M[\mu_1]$& $(1, 0, 4, 0, 84, 0, 2560, 0, 95060, \ldots)$\\
\hline
$M[\mu_2]$&$(1, 2, 22, 272, 4526, 83312, 1654888, 34565120, 749630638, \ldots)$\\
\hline
$M[\mu_3]$&$(1, 0, 52, 0, 46980, 0, 68708320, 0, 124144318340, \ldots)$\\
\hline
$M[\mu_4]$&$(1, 4, 82, 1780, 72830, 2755004, 138374800, 6709682816, 367576656446,  \ldots)$\\
\hline
\end{tabular}}
\caption{Moment Statistics for $y^3=x^{5}-1$ over $\mathbb{Q}(\zeta_5)$.}\label{table:moments35Q5}
\end{table}

\begin{table}[h]
{\renewcommand{\arraystretch}{1.5}
\begin{tabular}{|c|l|}
\hline
$M[\mu_1]$& $(1, 0, 8, 0, 168, 0, 5120, 0, 190120, \ldots)$\\
\hline
$M[\mu_2]$&$(1, 4, 40, 544, 9016, 166624, 3309376, 69130240, 1499256376, \ldots)$\\
\hline
$M[\mu_3]$&$(1, 0, 104, 0, 93960, 0, 137416640, 0, 248288636680, \ldots)$\\
\hline
$M[\mu_4]$&$(1, 6, 156, 3528, 145512, 5509296, 276746016, 13419347136, 735153215448,  \ldots)$\\
\hline
\end{tabular}}
\caption{Moment Statistics for $y^3=x^{5}-1$ over $\mathbb{Q}(\zeta_{15})$.}\label{table:moments35Q15}
\end{table}

\end{example}

\subsection{Tables of $\mu_1$- and $a_1$-moment statistics}\label{sec:a1mu1Moments}

Table \ref{table:momemntsa1} gives both $M[\mu_1]$ and the numerical  moments  of the normalized $L$-polynomial of some Catalan Jacobians (over $\mathbb{Q}$). The numerical moments were computed for primes up  to $2^{N}$, where $N$ is given in the last column of the table, using an algorithm described in \cite{HarveySuth2014} and \cite{HarveySuth2016}. As the genus grows, the computations require more processing power, and so a smaller bound $N$ was used. This leads to less accurate estimates for the numerical moments in higher genus, which make it more difficult to compare them to the theoretical moments obtained from the Sato-Tate groups. 

In Proposition \ref{prop:charpolya1} we proved that only matrices in the identity component have characteristic polynomials with nonzero $b_1$-coefficient. Let $M_n[{}^0\mu_1]$ denote the $n^{th}$ moment of the identity component. As noted in the proof of Proposition \ref{prop:charpolya1}, the $b_1$-coefficient of the characteristic polynomial is the trace of a random matrix $U\in\U(1)^g$. Thus, $M_n[{}^0\mu_1]=M_n[\mu_{\U(1)^g}]$ and we can use the formula in Equation \eqref{eqn:MomentsU1g}. To compute $M_n[\mu_1]$, we average over the components by dividing $M_n[\mu_{\U(1)^g}]$ by the size of the Galois group.

\begin{center}
\begin{table}[h]
\begin{tabular}{|c|c|l|l|l|l|l|l|}
\hline
$(p,q)$ & $g$& & $M_2$ & $M_4$ & $M_6$ & $M_8$ &N\\ 
\hline
(5,3)&4& $\mu_1$ & 1&21&640&23765 &  \\
&&  $a_1$  &    0.994  &  20.726  & 625.606  &  22961.668  & 23\\
\hline
(7,3)&6& $\mu_1$ &  1&33&1660&106785  &  \\
&&  $a_1$ & 0.998  &  32.718  &  1628.656 & 103534.931  & 23  \\
\hline
(11,3)&10& $\mu_1$ &  1 & 57& 5140&615545 &  
 \\
&&   $a_1$ &  0.972  &  55.306 & 4836.457  &  532489.325 & 19\\
\hline
(13,3)&12& $\mu_1$ &1&69&7600& 1121925 &   \\
&&   $a_1$ &  0.998 &  68.164  &  7570.201  & 1178539.197 &19   \\
\hline
(7,5)&12& $\mu_1$ &1&69&7600& 1121925 &   \\
&&   $a_1$ &  0.998 &  68.164  &  7570.201  & 1178539.197 &19   \\
\hline
(17,3)&16& $\mu_1$ &1&93& 13960&2840285 &   \\
&&   $a_1$&   0.954  &  91.335  &  15348.411 & 3864179.550  &  18   \\
\hline
\end{tabular}
\caption{Table of some $\mu_1$- and $a_1$-moments $\Jac(C_{p,q})$ over $\mathbb Q$.}\label{table:momemntsa1}
\end{table}
\end{center}

Note that $M_2[\mu_1]$ equals 1 for each example in Table \ref{table:momemntsa1}. We prove that this is true for every Catalan Jacobian in Proposition \ref{prop:secondmoment}.

\begin{remark}
In many cases, the $\mu_1$-moments of Catalan Jacobians match those of the Fermat quotients considered in \cite{FiteGonzalezLario2016}. This occurs when the genera are the same and the identity components are the same (some of the identity components in \cite{FiteGonzalezLario2016} are of the form $(\U(1)_{g/3})^{3}$), where $\U(1)_n=\left\langle \diag(\underbrace{u, \overline u, \ldots, u,\overline u}_{n-\text{times}}): u\in \mathbb C^\times, |u|=1\right\rangle$. These Jacobian varieties have a similar property to ours: only the identity components of the Sato-Tate groups contribute to the $b_1$-coefficients of the characteristic polynomials. However we expect to see a difference in the moments of the higher traces. 

For example, consider the two genus 6 curves $y^3=x^7-1$ (Catalan curve) and $v^{13}=u(u+1)^{10}$ (a curve from \cite{FiteGonzalezLario2016}) over $\mathbb{Q}$. Both of their associated Sato-Tate groups have identity component $\U(1)^6$, but their component groups are different. We can see this difference in the $\mu_2$-moments:
\begin{align*}
    y^3=x^7-1\colon &\;\;M[\mu_2]= (1,1,12, 206, 5796,\ldots),\\
    v^{13}=u(u+1)^{10}\colon &\;\;M[\mu_2]=(1, 1, 11, 206, 5781,\ldots). 
\end{align*}
\end{remark}

\subsection{General results for moment statistics}

\begin{proposition}\label{prop:secondmoment}
Over $\mathbb{Q}$, the second moment $M_2[\mu_1]$ equals 1 for every Catalan Jacobian. Over $\mathbb{Q}(\zeta_{p}),\mathbb{Q}(\zeta_{q})$, and $\mathbb{Q}(\zeta_{pq})$, the moment $M_2[\mu_1]$ is $p-1$, $q-1$, and $(p-1)(q-1)$, respectively.
\end{proposition}

\begin{proof}
In Proposition \ref{prop:charpolya1} we proved that only the identity component contributes to the $b_1$-coefficient of the characteristic polynomial.  We will demonstrate that the value obtained for the second moment $M_2[{}^0\mu_1]$ is $(p-1)(q-1)$, so that when averaging over all components we obtain the desired results.

This amounts to computing the sum in Equation \eqref{eqn:MomentsU1g} when $n=2$. In this situation, we are summing over $\beta_i$ values that add to $2$. There are two types of summands to consider: ones where there are only two nonzero values $\beta_i=\beta_j=1$ for some $i\not=j$ and ones where there is only one nonzero value $\beta_i=2$ for some $i$.

The first type does not contribute to the sum because $M_{1}[\mu_{\U(1)}]=0$ and, hence the entire term equals 0. On the other hand, each term of the second type will simplify to 
$$\binom{2}{2,0,\ldots,0}M_{2}[\mu_{\U(1)}]M_{0}[\mu_{\U(1)}]\cdots M_{0}[\mu_{\U(1)}]=2$$
since $M_{0}[\mu_{\U(1)}]=1$ and $M_{2}[\mu_{\U(1)}]=2$. There are exactly $g$ terms of this form since this is essentially counting the number of $g$-tuples with a single nonzero entry. Hence, $M_n[{}^0\mu_1]=2g=(p-1)(q-1)$.

We now consider the size of the component group for each of the subfields of $\mathbb{Q}(\zeta_{pq})$. Over $\mathbb{Q}$, the component group has size $(p-1)(q-1)$. Over $\mathbb{Q}(\zeta_p)$, $\mathbb{Q}(\zeta_q)$, $\mathbb{Q}(\zeta_{pq})$, the sizes are $q-1$, $p-1$, and 1, respectively.  Averaging $M_n[{}^0\mu_1]$ over these values yields the desired results.

\end{proof}

\section{Galois Endomorphism Types and Related Results}\label{sec:GaloisType}
We begin this section with some notation and terminology.
Let $A/F$ be an abelian variety defined over a number field $F$, and let $K$ be the endomorphism field of $A$. We will denote the real endomorphism algebra of $A$ by $\End(A_F)_{\mathbb R} :=\End(A_F)\otimes_{\mathbb{Z}} \mathbb{R}.$  Let $\mathbb H$ denote the quaternions and let $M_n(R)$ denote the $n\times n$ matrix ring over a ring $R$. Wedderburn's structure theorem tells us that the $\mathbb{R}$-algebra $\End(A_F)_\mathbb{R}$  satisfies
\begin{align}\label{eqn:EndoAlg_factor}
    \End(A_F)_\mathbb{R}\simeq \prod_i M_{t_i}(\mathbb{R})\times \prod_i M_{n_i}(\mathbb H)\times  \prod_i M_{p_i}(\mathbb{C}),
\end{align}
for some nonnegative integers $t_i, n_i, p_i$.

Let $\mathcal C$ be the category of pairs $(G,E)$, where $G$ is a finite group and $E$ is an $\mathbb{R}$-algebra equipped with an $\mathbb{R}$-linear action of $G$. The \textbf{Galois endomorphism type} of $A/F$ is the isomorphism class in $\mathcal C$ of the pair $[\Gal(K/F), \End(A_F)_\mathbb{R}]$ \cite[Definition 1.3]{FKS2012}. If $L\subseteq K$ is an intermediate field corresponding to a subgroup $N\subseteq \Gal(K/F)$, then $\End(A_L)_\mathbb{R}\simeq (\End(A_K)_\mathbb{R})^N$ (see, for example, \cite[Section 6]{FS2016}). In this section, we will prove the following result for the Galois endomorphism types of Catalan Jacobians.

\begin{theorem}\label{thm:EndoAlgebraSubfields}
Let $K=\mathbb{Q}(\zeta_{pq})$ and let $L$ be any intermediate field $\mathbb{Q}\subseteq L\subseteq K$. Then
$$\End(\Jac(C_{p,q})_L)_{\mathbb{R}}=
\begin{cases}
\mathbb{R}^{[L:\mathbb{Q}]}&\text{ if } L\subseteq \mathbb{R},\\
\mathbb{C}^{[L:\mathbb{Q}]/2}&\text{ otherwise}. 
\end{cases}
$$
\end{theorem}

Additionally, we will prove the following result regarding the N\'eron-Severi group $\NS(\Jac(C_{p,q})_L)$ of the Catalan Jacobian. 

\begin{theorem}\label{thm:NeronSeveri}
For any intermediate field $\mathbb{Q}\subseteq L\subseteq K$, the rational N\'eron-Severi group of the Catalan Jacobian is $\NS(\Jac(C_{p,q})_L)_{\mathbb{Q}}\simeq \mathbb{Q}^r$, where 
$$r=
\begin{cases}
[L:\mathbb{Q}]&\text{ if } L\subseteq \mathbb{R},\\
\frac12 [L:\mathbb{Q}]&\text{ otherwise}. 
\end{cases}
$$
\end{theorem}

We will prove these results using two methods: using data obtained from moment statistics and working with Rosati forms. The latter is the more traditional method of determining real endomorphism algebras, whereas the former is a new technique using recent results of Costa, Fit\'e, and Sutherland \cite{Costa2019}.

\subsection{Data obtained from moment statistics}\label{sec:EndoAlg_Moments}

We start with some notation. The Sato-Tate group of $\Jac(C_{p,q})$ is equipped with a faithful, self-dual representation $\rho\colon \ST(\Jac(C_{p,q}))\rightarrow \GL(V)$, where $V$ is a $2g$-dimensional $\mathbb{C}$-vector space. We can use this representation to view $\ST(\Jac(C_{p,q}))$ as a compact real Lie subgroup of $\USp(2g)$. From here, we define the following virtual characters of the Sato-Tate group of $\Jac(C_{p,q})$:
$$a_1=\Tr(V),\;\; a_2=\Tr(\wedge^2 V),\;\; s_2=a_1^2-a_2.$$
We will be interested in certain moments of these characters: $M_2[a_1], M_1[a_2],$ and $M_1[s_2]$. These moments satisfy the following equations
$$M_2[a_1]=M_2[\mu_1],\;\; M_1[a_2]=M_1[\mu_2],\;\; M_1[s_2]=M_2[a_1]-2M_1[a_2],$$
where $\mu_1$ and $\mu_2$ are the measures defined in Section \ref{sec:MomentsIntro}.

The moment $M_1[s_2]$  can be interpreted as a Frobenius-Schur indicator of the standard representation of $\ST(A)$ (see the exposition in Remark 5 of \cite{Costa2019} for more details). For an irreducible representation, this value can only be $-1, 0,$ or 1 (see Proposition 39 of \cite{Serre1977}).

We can use the results of \cite{Costa2019} to obtain further identities for these moments. Proposition 1 of \cite{Costa2019} states that, for any  abelian variety $A$ defined over a field $k$,
$$M_2[a_1]=\rk_\mathbb{Z}(\End(A_k)).$$

\begin{proposition}\label{prop:endoring}
The real endomorphism algebra $\End(\Jac(C_{p,q})_\mathbb{Q})_\mathbb{R}$ for any Catalan Jacobian is $\mathbb{R}$.
\end{proposition}

\begin{proof} 
We proved in Proposition \ref{prop:secondmoment} that $M_2[\mu_1]=1$. Thus, $\rk_\mathbb{Z}(\End(\Jac(C_{p,q})_\mathbb{Q}))=1$, and the only possibility for $\End(\Jac( C_{p,q})_\mathbb{Q})_\mathbb{R}$ is $\mathbb{R}$.
\end{proof}

\begin{corollary}
The rank of the N\'eron-Severi group of $\Jac(C_{p,q})_\mathbb{Q}$ is 1.
\end{corollary}
\begin{proof}
This follows from Proposition \ref{prop:endoring} and the fact that the N\'eron-Severi group of a principally polarized abelian variety $A$ embeds into $\End(A)\otimes_\mathbb{Z} \mathbb{Q}$. 

Alternatively, we can use Lemma 3 of \cite{Costa2019} to prove the result.  Applying this lemma to our situation yields the following equation:
$$1-2\cdot \rk_\mathbb{Z}(\NS(\Jac(C_{p,q})_\mathbb{Q}))=-1.$$
Hence, $\rk_\mathbb{Z}(\NS(\Jac(C_{p,q})_\mathbb{Q}))=1.$
\end{proof}
Proposition 2 of \cite{Costa2019} gives an additional identity for the the rank of $\NS(A_F)$:
\begin{align*}\label{eqn:rankNS}
M_1[a_2]=\rk_\mathbb{Z}(\NS(A_F)).    
\end{align*}
Thus, we obtain the following corollary.
\begin{corollary}\label{cor:M1a2moment}
For the Catalan Jacobian $\Jac(C_{p,q})_\mathbb{Q}$, we have
$$M_1[a_2]=1.$$
\end{corollary}

\begin{corollary}\label{cor:M1s2moment}
Over $\mathbb{Q}$, the  moment $M_1[s_2]$ equals $-1$ for every Catalan Jacobian. Over $\mathbb{Q}(\zeta_{p}),\mathbb{Q}(\zeta_{q})$, and $\mathbb{Q}(\zeta_{pq})$, the moment is $M_1[s_2]=0$.
\end{corollary}

\begin{proof}
Over $\mathbb{Q}$, we have already proved that $M_2[a_1]=1$ and $M_1[a_2]=1$ in Proposition \ref{prop:secondmoment} and Corollary \ref{cor:M1a2moment}. Thus, $M_1[s_2]=1-2=-1$ as desired.
\end{proof}

In some cases, we can use the results of \cite{Costa2019} to uniquely determine the real endomorphism algebras and ranks of N\'eron-Severi groups. For example, this seems to be possible in dimension 2 and 3 (see the tables of moment statistics in \cite{FKS2012, FKS2021satotate}). However, this is not the case for higher genus Catalan Jacobians. In order to further classify real endomorphism algebras, we will use Rosati forms.

\subsection{Further information from Rosati forms}\label{sec:EndoAlg_Rosati}

Every polarization on an abelian variety induces an anti-involution, called the Rosati involution (see Chapter 5 of \cite{LangeBirkenhake}). Theorem 5.5.6 of \cite{LangeBirkenhake}, combined with the fact that Catalan Jacobians are Type IV in the Albert's classification of abelian varieties, tells us that the Rosati involution on $\End(\Jac(C_{p,q})_K)_{\mathbb{Q}}$ corresponds to complex conjugation on $\mathbb{Z}[\zeta_{pq}]$. We now complete the proofs of Theorems \ref{thm:EndoAlgebraSubfields} and \ref{thm:NeronSeveri}.

\begin{proof}[Proof of Theorem \ref{thm:NeronSeveri}]

For any intermediate field $\mathbb{Q}\subset L\subseteq K$, the symmetric elements of $\End(\Jac(C_{p,q})_L)_{\mathbb{Q}}$ (i.e., those fixed by the Rosati involution) correspond to the elements of $L$ fixed by complex conjugation. Proposition 5.2.1 of \cite{LangeBirkenhake} gives an isomorphism between the set of symmetric elements and the rational N\'eron-Severi group of an abelian variety. Applying this to our situation yields the desired result.

\end{proof}

\begin{proof}[Proof of Theorem \ref{thm:EndoAlgebraSubfields}]
Either the intermediate field $L$ is totally real or it is not. When $L$ is totally real, the Rosati involution is trivial on $\End(\Jac(C_{p,q})_L)\simeq \mathcal O_L$ since it acts as complex conjugation. Hence, $\End(\Jac(C_{p,q})_L)_{\mathbb{R}}$ equals $\mathbb{R}^{[L:\mathbb{Q}]}$. 

Otherwise, by Theorem 5.5.6 of \cite{LangeBirkenhake}, the rational endomorphism algebra of $\Jac(C_{p,q})_L$ admits a positive anti-involution of the second kind (the Rosati involution). The center of $\End(\Jac(C_{p,q})_L)_{\mathbb{Q}}$ is a totally complex quadratic extension of a totally real field. In fact, in our setting, the center equals $\End(\Jac(C_{p,q})_L)_{\mathbb{Q}}$. Thus, by Proposition 5.5.7 of \cite{LangeBirkenhake}, the real endomorphism algebra $\End(\Jac(C_{p,q})_L)_{\mathbb{R}}$ equals $\mathbb{C}^{[L:\mathbb{Q}]/2}$.

\end{proof}

\begin{corollary}\label{cor:EndoAlg_pq}
The real endomorphism algebra $\End(\Jac(C_{p,q})_{K})_\mathbb{R}$ is $\mathbb{C}^g$. 
\end{corollary}
\begin{remark}
This result follows from Theorem \ref{thm:EndoAlgebraSubfields}, but it can also be proved by through computations with Rosati forms. Since $\Jac(C_{p,q})/K$ is nondegenerate, its complex endomorphism algebra  is the subspace of $M_{2g}(\mathbb{C})$ fixed by the action of the identity component of the Sato-Tate group, and the real endomorphism algebra  $\End(\Jac(C_{p,q})_{K})_\mathbb{R}$ is the subspace of half the dimension for which the Rosati form is positive definite (see Definition 2.18 in \cite{FKS2012}). 
\end{remark}

\begin{corollary}
For any intermediate field $L\subseteq K$, the Frobenius-Schur indicator is 
$$M_1[s_2]=\begin{cases}
-[L:\mathbb{Q}] &\text{ if $L$ is totally real},\\
0 &\text{ otherwise.}
\end{cases}$$
\end{corollary}

\begin{proof}
Recall that $M_1[s_2]=M_2[a_1]-2M_1[a_2]$. Propositions 1 and 2 of \cite{Costa2019} show that $M_2[a_1]=\rk_\mathbb{Z}(\End(\Jac(C_{p,q})_L))$ and  $M_1[a_2]=\rk_\mathbb{Z}(\NS(\Jac(C_{p,q})_L))$. The results of Theorem \ref{thm:EndoAlgebraSubfields} tell us that the rank of the endomorphism ring is $[L:\mathbb{Q}]$ for any intermediate field $L$. Combining this with the value of the rank of the N\'eron-Severi group from Theorem \ref{thm:NeronSeveri} yields the result.
\end{proof}

%%%%%%%%%%%%%%%%%%%%%%%%%%%%%%%%%%%%%%%%%%%%%%%
%%%%%%%%%%%%%%%%%%%%%%%%%%%%%%%%%%%%%%%%%%%%%%%

\section*{Acknowledgements}

Many thanks to Francesc Fit\'e, Edgar Costa, and Drew Sutherland while working on Section \ref{sec:EndoAlg_Moments}. Thanks also to Mckenzie West for patiently answering my Sage and GitHub questions while I was working on Table \ref{table:momemntsa1}. I would  like to thank Drew Sutherland for sharing access to a server at MIT while I was computing  the numerical moments given in Table \ref{table:momemntsa1}. 

I am grateful for the support and encouragement of Matt Montesano, Alanna Hoyer-Letizel, and the Rethinking Number Theory Workshop co-organizers, project leaders, and participants. They provided a sense of community during a difficult time, and it would not have been possible to write this article during a pandemic without them.

My research is supported by a PSC-CUNY Award, jointly funded by The Professional Staff Congress and The City University of New York.

% \section*{Declarations}

% Some journals require declarations to be submitted in a standardised format. Please check the Instructions for Authors of the journal to which you are submitting to see if you need to complete this section. If yes, your manuscript must contain the following sections under the heading `Declarations':

% \begin{itemize}
% \item Funding
% \item Conflict of interest/Competing interests (check journal-specific guidelines for which heading to use)
% \item Ethics approval 
% \item Consent to participate
% \item Consent for publication
% \item Availability of data and materials
% \item Code availability 
% \item Authors' contributions
% \end{itemize}

% \noindent
% If any of the sections are not relevant to your manuscript, please include the heading and write `Not applicable' for that section. 

%%===================================================%%
%% For presentation purpose, we have included        %%
%% \bigskip command. please ignore this.             %%
%%===================================================%%
% \bigskip
% \begin{flushleft}%
% Editorial Policies for:

% \bigskip\noindent
% Springer journals and proceedings: \url{https://www.springer.com/gp/editorial-policies}

% \bigskip\noindent
% Nature Portfolio journals: \url{https://www.nature.com/nature-research/editorial-policies}

% \bigskip\noindent
% \textit{Scientific Reports}: \url{https://www.nature.com/srep/journal-policies/editorial-policies}

% \bigskip\noindent
% BMC journals: \url{https://www.biomedcentral.com/getpublished/editorial-policies}
% \end{flushleft}

\begin{appendices}

\section{Examples of the component group generators}\label{app:component}

In Table \ref{table:gammas} we give examples of the matrices $\gamma_q$ and $\gamma_p$  from Definitions \ref{def:gamma_q} and \ref{def:gamma_p}. These were computed in Sage \cite{Sage} using Sage's chosen generators for   $(\mathbb Z/q\mathbb Z)^{\times}$ and  $(\mathbb Z/p\mathbb Z)^{\times}$.
{
\begin{table}[h]
\begin{tabular}{|c|c|c|}
\hline
$(p,q)$&  $\gamma_p$& $\gamma_q$ \\ 
\hline
(5,3)& $\begin{pmatrix}0&0&0&J\\0&0&I&0\\J&0&0&0\\0&I&0&0 \end{pmatrix}$ & $\begin{pmatrix}0&I&0&0\\I&0&0&0\\0&0&0&J\\0&0&J&0\end{pmatrix}$\\
\hline
(7,3)& $\begin{pmatrix}0&0&0&0&0&J\\0&0&J&0&0&0\\0&0&0&0&I&0\\J&0&0&0&0&0\\0&0&0&I&0&0\\0&J&0&0&0&0 \end{pmatrix}$ & $\begin{pmatrix}0&0&I&0&0&0\\0&0&0&I&0&0\\I&0&0&0&0&0\\0&I&0&0&0&0\\0&0&0&0&0&J\\0&0&0&0&J&0 \end{pmatrix}$\\
\hline
\end{tabular}
\caption{Examples of $\gamma_p$ and $\gamma_q$ matrices for $y^q=x^p-1$.}\label{table:gammas}
\end{table}

}
\end{appendices}

\bibliographystyle{plain}
\bibliography{Catalanbib.bib}
\end{document}